\documentclass[12pt]{amsart}
\usepackage{amssymb}
\usepackage{verbatim}
\usepackage[usenames]{color}
\usepackage{hyperref}
\usepackage{url}
\usepackage[all]{xy}
\newtheorem{theorem}{Theorem}[section]
\newtheorem{proposition}[theorem]{Proposition} 
\newtheorem{lemma}[theorem]{Lemma}

\theoremstyle{definition}
\newtheorem{definition}[theorem]{Definition}

\newtheorem{question}[theorem]{Question}
\newtheorem{observation}[theorem]{Observation}

\theoremstyle{remark}
\newcommand{\defemph}{\textit}

\newcommand{\zfc}{\textrm{ZFC}}
\newcommand{\zf}{\textrm{ZF}}

\newcommand{\mc}{\mathcal}
\newcommand{\mbb}{\mathbb}

\newcommand{\p}{\mathcal{P}}

\title[Ramsey Theory on Generalized Baire Space]
 {Ramsey Theory on Generalized Baire Space}
\author{Dan Hathaway}
\address{Mathematics Department\\
University of Denver\\
Denver, CO 80208, U.S.A.}
\email{Daniel.Hathaway@du.edu}
\begin{document}

\begin{abstract}
We show that although
 the Galvin-Prikry Theorem
 does not hold on generalized Baire
 space with the standard topology,
 there are similar theorems which do hold
 on generalized Baire space
 with certain coarser topologies.
\end{abstract}

\maketitle

\section{Introduction and Some Definitions}

Given ordinals $\kappa$ and $\gamma$,
 $[\kappa]^\gamma$ is the set of all
 subsets of $\kappa$ of order type $\gamma$,
 and $[\kappa]^{<\gamma}$ is the set of all
 subsets of $\kappa$ of order type $< \gamma$.
In this paper, for an infinite cardinal $\kappa$,
 we will consider colorings of $[\kappa]^\kappa$
 as opposed to $[\kappa]^\mu$ for some $\mu < \kappa$.
Given a function $c : X \to Y$
 and a set $Z \subseteq X$,
 $c``Z$ is the image of $Z$ under $c$.
We use the convention that natural numbers
 are ordinals, so for example $2 = \{0,1\}$.
We will sometimes use the notation
 $(\alpha,\beta)$ for the set of all ordinals
 $\gamma$ such that $\alpha < \gamma < \beta$,
 and $(\alpha,\beta]$ for the set
 $(\alpha,\beta) \cup \{\beta\}$, etc.

\begin{definition}
Let $\kappa$ be a cardinal.
Given sets $A, B \subseteq \kappa$,
 a pair $(A,B)$
 such that $A \cap B = \emptyset$
 is called a \defemph{pattern}.
Given $\mc{A}, \mc{B} \subseteq \p(\kappa)$,
 an $(\mc{A}, \mc{B})$-\defemph{pattern} is a pair
 $(A,B)$ such that $A \in \mc{A}$ and $B \in \mc{B}$.
A set $X \in [\kappa]^\kappa$ \defemph{matches}
 the pattern $(A,B)$ iff
 $A \subseteq X$ and $B \cap X = \emptyset$.
Finally,
 $[A;B]$ is the set of all $X \in [\kappa]^\kappa$
 which match $(A,B)$.
\end{definition}

\begin{definition}
Fix $\mc{A}, \mc{B} \subseteq \p(\kappa)$.
$\Sigma(\mc{A}, \mc{B})$
 is the collection of all
 $\mc{S} \subseteq [\kappa]^\kappa$
 that are unions of sets of the form
 $[A;B]$ for $(A,B) \in \mc{A} \times \mc{B}$.
That is, sets $\mc{S}$
 for which there exists a set $\mc{Q}$ of
 $(\mc{A},\mc{B})$-patterns
 such that
 $\mc{S} = \{ X
 \in [\kappa]^\kappa : X \mbox{ matches some }
 (A,B) \in \mc{Q} \}$.
We say that $\mc{Q}$ \defemph{generates}
 $\mc{S}$.
$\Delta(\mc{A}, \mc{B})$ is the collection of all
 $\mc{S} \subseteq [\kappa]^\kappa$ such that
 $\mc{S}$ and $[\kappa]^\kappa - \mc{S}$ are in
 $\Sigma(\mc{A},\mc{B})$.
\end{definition}

Hence, $\mc{S} \in \Sigma(\mc{A},\mc{B})$ iff
 there is a collection of patterns
 $\{ (A_i, B_i) \in \mc{A} \times \mc{B} :
 i \in I \}$ such that
 for each $X \in [\kappa]^\kappa$,
 $X \in \mc{S}$ iff
 $(\exists i \in I)$
 $X$ matches $(A_i,B_i)$.
Also, $\mc{S} \in \Delta(\mc{A}, \mc{B})$
 iff there are sets $\mc{Q}^+, \mc{Q}^-$
 of $(\mc{A},\mc{B})$-patterns
 such that for each $X \in [\kappa]^\kappa$,
 $X \in \mc{S}$ iff $X$ matches some $(A,B) \in \mc{Q}^+$,
 and $X \not\in \mc{S}$ iff $X$ matches some
 $(A,B) \in \mc{Q}^-$.

If $\mc{A}$ and $\mc{B}$ are closed under finite unions,
 then $\Sigma(\mc{A},\mc{B})$ is a topology:
 it is closed under finite intersections
 and arbitrary unions,
 and has both $\emptyset$ and $[\kappa]^\kappa$ as elements.
If $\Sigma(\mc{A},\mc{B})$ is a topology,
 then $\Delta(\mc{A}, \mc{B})$ is the collection of
 clopen sets in this topology.
$\Sigma([\kappa]^{<\kappa}, [\kappa]^{<\kappa})$
 is the standard topology on generalized Baire space
 of height $\kappa$.
%Let $\mc{M}_{A,B} \subseteq [\kappa]^\kappa$
% be the set of $X \in [\kappa]^\kappa$ that match
% the pattern $(A,B)$.
%We would like to think of each $\mc{M}_{A,B}$
% as being a basic open set.
%However, the collection of these
% sets $\mc{M}_{A,B}$ may not form a basis.
%In the event that they do,
% the $\Gamma(\mc{A},\mc{B})$
% defined below is the collection of open sets
% in the topology with the basis
% $$\{ \mc{M}_{A,B} :
% (A,B) \mbox{ is an }
% (\mc{A}, \mc{B})\mbox{-pattern} \}.$$
%

\begin{definition}
A collection $\mc{S} \in [\kappa]^\kappa$
 is \defemph{Ramsey as witnessed by}
 $H \in [\kappa]^\kappa$
 iff one of the following holds:
\begin{itemize}
\item[1)] $(\forall X \in [H]^\kappa)\, X \in \mc{S}$;
\item[2)] $(\forall X \in [H]^\kappa)\, X \notin \mc{S}$.
\end{itemize}
We also say that $H$ is \defemph{homogeneous}
 for $\mc{S}$.
%A collection $\mc{S}$ is \defemph{Ramsey}
% iff it is Ramsey as witnessed by some set
% $H \in [\kappa]^\kappa$.
More generally, we say that
 $c : [\kappa]^\kappa \to \lambda$ is Ramsey just in case
 there is a set $H \in [\kappa]^\kappa$
 such that $|c``[H]^\kappa| = 1$,
 and we say that $H$ is homogeneous for $c$.
\end{definition}

%One of the earliest results in this area is
% due to Nash-Williams (VERIFY AND CITE!!!), who proved that every
% $\mc{S} \in \Delta([\omega]^{<\omega}, [\omega]^{<\omega})$
% is Ramsey.
One of the earliest results in this area is
 the Galvin-Prikry Theorem \cite{gp},
 which says that not only is every
 open set in the topology
 $\Sigma([\omega]^{<\omega}, [\omega]^{<\omega})$ Ramsey,
 but every Borel set in this topology is Ramsey as well.
Next, Silver \cite{silver} showed that every analytic
 set in the topology $\Sigma([\omega]^{<\omega}, [\omega]^{<\omega})$ is Ramsey.
Ellentuck generalized this further \cite{ellentuck}
 by showing that every
 analytic $\mc{S}$ in the topology
 $\Sigma([\omega]^{<\omega}, [\omega]^{\le \omega})$
 is Ramsey.
Assuming the Axiom of Choice, there exists a set
 $\mc{S} \subseteq [\omega]^\omega$ that is not Ramsey.
Moreover,
 Silver \cite{silver} showed that it is consistent with $\zfc$
 that there is a logically simple,
 in fact ${\Delta}^1_2$, set $\mc{S} \subseteq [\omega]^\omega$
 that is not Ramsey.
On the other hand \cite{kan}, if we assume the existence
 of large cardinals, then every $\mc{S} \subseteq [\omega]^\omega$
 that is in $L(\mbb{R})$ is Ramsey,
 where $L(\mbb{R})$ is the smallest model of $\zf$
 that contains $\mbb{R}$ and all the ordinals.
Let us also mention that Shelah \cite{sh}
 has shown that if $\kappa$ is a Ramsey cardinal
 and $c : [\kappa]^\omega \to 2$ is Borel
 in a certain topology,
 then there is a set $H \in [\kappa]^\kappa$
 such that $|c``[H]^\omega| = 1$.

It is natural to ask what sets
 $\mc{S} \subseteq [\kappa]^\kappa$ for $\kappa > \omega$
 are Ramsey.
The standard argument that there is a set
 $\mc{S} \subseteq [\omega]^\omega$ that is not Ramsey
 shows that when $\kappa > \omega$,
 there is a set $\mc{S} \subseteq [\kappa]^{\kappa}$
 in $\Delta([\kappa]^{\omega}, [\kappa]^{<\kappa})$
 that is not Ramsey
 (see Proposition~\ref{acmakesitfail}).
In Section~\ref{seciton_coarse}
 we make the main contribution of this paper
 and show that when $\gamma < \kappa$,
 then all $\Delta(
 [\kappa]^{<\gamma},
 [\kappa]^{<\gamma})$ sets are Ramsey.
It is open whether $\Delta$ can be replaced with $\Sigma$.

Then, when we increase the $\mc{B}$ component of the
 patterns to include all size $<\kappa$ sets,
 we must simultaneously decrease the $\mc{A}$ component.
In Section~\ref{section_wc}, we show that the following
 are equivalent for a cardinal $\kappa > \omega$:
\begin{itemize}
\item $\kappa$ is weakly compact;
\item All $\Delta([\kappa]^2, [\kappa]^{<\kappa})$
 sets are Ramsey;
\item All $\Sigma([\kappa]^2, [\kappa]^{<\kappa})$
 sets are Ramsey;
\item $(\forall n \in \omega)$ all
 $\Sigma([\kappa]^n, [\kappa]^{<\kappa})$
 sets are Ramsey;
\end{itemize}
The main technique of the section
 is a shrinking procedure.
Here is the basic version:
 fix a set
 of $([\kappa]^2, [\kappa]^{<\kappa})$-patterns
 $\mc{Q}$ and a set
 $H \in [\kappa]^\kappa$ such that for each
 $A \in [H]^2$, there is some $B_A$ such that
 $(A,B_A) \in \mc{Q}$.
Then there is some
 $H' \in [H]^\kappa$ such that for all distinct
 $a_1, a_2 \in H'$, the first element of $H'$
 greater than $a_1$ and $a_2$ is also
 greater than all elements of
 $B_{\{a_1, a_2\}}$.
Each $X \in [H']^\kappa$
 will match $(A,B_A)$, where $A$ is the set of
 the first two elements of $X$.
We will modify this procedure in the following section.

In Section~\ref{section_ramsey},
 we strengthen the $\mc{A}$ component of the patterns and
 show that
 if $\kappa$ is a Ramsey cardinal, then all
 $\Sigma([\kappa]^{<\omega}, [\kappa]^{<\kappa})$
 sets are Ramsey.
In Section~\ref{section_measurable},
 we strengthen the $\mc{B}$ component of the patterns and
 show that
 if $\kappa$ is a measurable cardinal with a $\kappa$-complete
 ultrafilter $\mc{U}$, then all
 $\Sigma([\kappa]^{<\omega}, \mc{P}(\kappa) - \mc{U})$ sets
 are Ramsey.
Finally, in Section~\ref{section_constructible}
 we consider sets of patterns that are within $L$,
 assuming $0^\#$ exists.

\section{All $\Delta(
 [\kappa]^{<\gamma},
 [\kappa]^{<\gamma})$
 sets are Ramsey if $\gamma < \kappa$}
\label{seciton_coarse}

Temporarily fix cardinals $\gamma < \kappa$.
We call $\Sigma([\kappa]^{<\gamma}, [\kappa]^{<\gamma})$
 the ${{<}\gamma}$-\textit{box topology};
 it is indeed a topology,
 and basic open sets are ``boxes'' determined by
 specifying membership requirements for ${<}\gamma$
 elements of $\kappa$.
We have that
 $$\Sigma([\kappa]^{<\gamma}, [\kappa]^{<\gamma}) \subseteq
   \Sigma([\kappa]^{<\kappa}, [\kappa]^{<\kappa}).$$
It turns out that because
 $\Sigma([\kappa]^{<\gamma}, [\kappa]^{<\gamma})$ is so coarse,
 all
 $\Delta([\kappa]^{<\gamma}, [\kappa]^{<\gamma})$ sets are Ramsey.
This follows from the next theorem:

\begin{theorem}
\label{theorem2}
Let $\gamma < \kappa$ be infinite cardinals.
Let $c : [\kappa]^\kappa \to \gamma$ be continuous,
 where $[\kappa]^\kappa$ is given the topology
 $\Sigma([\kappa]^{<\gamma}, [\kappa]^{<\gamma})$
 and $\gamma$ is given the discrete topology.
%Then each $\Delta([\kappa]^{<\gamma}, [\kappa]^{<\gamma})$ set
%is Ramsey as witnessed by some $H \in [\kappa]^\kappa$
Then there is some $H \in [\kappa]^\kappa$ that is
 homogeneous for $c$,
 where $|\kappa - H| \le \gamma$.
If $\gamma$ is a regular cardinal,
 we can get an $H$ such that
 $|\kappa - H| < \gamma$.
\end{theorem}
\begin{proof}
We will find a set $B \in [\kappa]^{\le \gamma}$
 such that $c \restriction [0;B]$ is constant.
If $\gamma$ is regular, we will have
 $|B| < \gamma$.
Let $\langle c_\alpha : \alpha < \gamma \rangle$
 be an enumeration of $\gamma$ where each ordinal
 is listed $\gamma$ times.
We will construct $A_\alpha, B_\alpha \in [\kappa]^{<\gamma}$
 for $\alpha < \gamma$ such that
 $A_\alpha \cap B_\alpha = \emptyset$ and the sets
 $A_\alpha$ are pairwise disjoint.
At stage $\alpha < \gamma$,
 let $B = \bigcup_{\beta < \alpha} A_\beta$.
Note that $|B| \le \gamma$, and if $\gamma$ is regular,
 then $|B| < \gamma$.
There are two possibilities.

Case 1.
If $c \restriction [\emptyset;B]$ is constantly $c_\alpha$,
 then terminate the construction.

Case 2.
Fix some $X \in [\emptyset;B]$ such that $c(X)
 \not= c_\alpha$.
Let $d_\alpha = c(X)$.
Since $c$ is continuous,
 fix disjoint $A_\alpha, B_\alpha$ such that
 $X \in [A_\alpha;B_\alpha]$ and
 $c \restriction [A_\alpha;B_\alpha]$ is constantly $d_\alpha$.
Note that since $A_\alpha \subseteq X$
 and $X \cap B = \emptyset$,
 $A_\alpha$ is disjoint from each
 $A_\beta$ for $\beta < \alpha$.
%If $\gamma$ is regular, we have $|B| < \gamma$.

We claim that the construction must terminate
 before stage $\gamma$.
Suppose that this is not the case.
Fix $X \in [\emptyset; \bigcup_{\alpha < \gamma} B_\alpha]$.
Fix disjoint $A,B \in [\kappa]^{<\gamma}$
 such that $X \in [A;B]$ and
 $c \restriction [A;B]$ is constantly $c(X)$.
Only $< \gamma$ many $A_\alpha$'s can intersect $B$,
 because the $A_\alpha$'s are pairwise disjoint.
Fix $\alpha < \gamma$ such that $A_\alpha$ is disjoint
 from $B$ and $c_\alpha = c(X)$.
Since $A \subseteq X$,
 $A$ is disjoint from $B_\alpha$.
We now have that
 $A$ and $A_\alpha$
 are each disjoint from $B$ and $B_\alpha$.
Thus,
 $(A \cup A_\alpha , B \cup B_\alpha)$ is a pattern.
We now have that $c$ is constantly $c(X)$ on
 $[A;B]$ and it is constantly
 $d_\alpha \not= c_\alpha = c(X)$
 on $[A_\alpha;B_\alpha]$.
But since
 $$[A \cup A_\alpha ; B \cup B_\alpha] \subseteq
 [A;B] \cap [A_\alpha;B_\alpha],$$
this is impossible.
\end{proof}

An important fact used in the proof above
 is that the coloring is
 $\Delta([\kappa]^{<\gamma}, [\kappa]^{<\gamma})$,
 as opposed to just
 $\Sigma([\kappa]^{<\gamma}, [\kappa]^{<\gamma})$.
We ask whether these more general sets are Ramsey:

\begin{question}
Let $\gamma < \kappa$ be infinite cardinals.
Is every
 $\Sigma([\kappa]^{<\gamma}, [\kappa]^{<\gamma})$
 set Ramsey?
In particular, is every
 $\Sigma([\omega_1]^2, [\omega_1]^1)$ set Ramsey?
If $\kappa$ is a measurable cardinal,
 is every $\Sigma([\kappa]^\omega, [\kappa]^1)$
 set Ramsey?
\end{question}

In the conclusion of the previous theorem,
 $H$ satisfies $|\kappa - H| \le \gamma$.
This allows us to simultaneously homogenize
 $< \kappa$ sets that are all
 $\Delta([\kappa]^{<\gamma}, [\kappa]^{<\gamma})$.

% \begin{cor}
% Let $\gamma < \kappa$ be infinite cardinals.
% Fix $\lambda < \kappa$.
% Let $c : [\kappa]^\kappa \to \lambda$ be a function
%  such that for each $\alpha < \lambda$,
%  $c^{-1}(\alpha)$ is
%  $\Delta([\kappa]^{<\gamma}, [\kappa]^{<\gamma})$.
% Then $c$ is Ramsey.
% \end{cor}
% \begin{proof}
% For each $\alpha < \lambda$,
%  let $c_\alpha : [\kappa]^{\kappa} \to 2$
%  be the function
%  $$c_\alpha(X) := \begin{cases}
%  1 & \mbox{if } c(X) = \alpha, \\
%  0 & \mbox{if } c(X) \not= \alpha.
%  \end{cases}$$
% Applying the previous theorem,
%  let $H_0 \in [\kappa]^\kappa$
%  be such that
%  $|c_0``([H_0]^{\kappa})| = 1$
%  and $|\kappa - H_0| \le \gamma$.
% Let $H_1 \in [H_0]^\kappa$
%  be such that
%  $|c_1``([H_1]^{\kappa})| = 1$
%  and $|H_0 - H_1| \le \gamma$.
% Continue like this,
%  taking intersections at limit stages.
% The construction will never get stuck
%  because $\gamma, \lambda < \kappa$.
% Let $H$ be the intersection of the entire
%  sequence.
% We have that
%  $|\kappa - H| \le \max\{ \gamma, \lambda \}$.
% For each $\alpha < \lambda$, we have
%  $|c_\alpha``([H]^\kappa)| = 1$.
% Thus,
%  $|c([H]^\kappa)| = 1$.
% \end{proof}

\section{All $\Sigma([\kappa]^2, [\kappa]^{<\kappa})$
 sets are Ramsey iff $\kappa$ is weakly compact}
\label{section_wc}

If $\kappa$ is not a weakly compact cardinal,
 then there is a coloring of $[\kappa]^2$ such that
 there is no $H \in [\kappa]^\kappa$ all of whose
 pairs are the same color.
The collection $\Sigma([\kappa]^2, [\kappa]^{<\kappa})$
 is fine enough to allow the following:
\begin{observation}
\label{fineenoughforpairs}
For each pair $\{ a_1, a_2 \} \in [\kappa]^2$,
 there is a $([\kappa]^2, [\kappa]^{<\kappa})$-pattern
 $(A,B)$ such that a set $X \in [\kappa]^\kappa$
 matches $(A,B)$ iff its first two elements are
 $a_1$ and $a_2$.
\end{observation}

This allows us to color a set $X \in [\kappa]^\kappa$
 based on its first two elements.

\begin{proposition}
Let $\kappa$ be an infinite cardinal that is not
 weakly compact.
Then there is a set in
 $\Delta([\kappa]^2, [\kappa]^{<\kappa})$
 that is not Ramsey.
\end{proposition}
\begin{proof}
Since $\kappa$ is not weakly compact,
 fix a coloring $c : [\kappa]^2 \to 2$
 such that there is no $H \in [\kappa]^2$
 satisfying $|c``[H]^2| = 1$.
Using the observation above,
 let $\mc{S} \in \Delta([\kappa]^2, [\kappa]^{<\kappa})$
 be the unique subset of $[\kappa]^\kappa$
 such that for each
 $X \in [\kappa]^\kappa$, we have
 $X \in \mc{S}$ iff
 $c(\{a_1, a_2\}) = 1$,
 where $a_1, a_2$ are the first two elements of $X$.
To see that $\mc{S}$ is indeed
 $\Delta([\kappa]^2, [\kappa]^{<\kappa})$,
 consider the first two elements $a_1, a_2$ of $X$.
If $c( \{ a_1, a_2 \} ) = 1$, then there is a
 $([\kappa]^2, [\kappa]^{<\kappa})$-pattern
 which witnesses that $X \in \mc{S}$.
If $c( \{ a_1, a_2 \} ) = 0$, then there is a
 $([\kappa]^2, [\kappa]^{<\kappa})$-pattern
 which witnesses that $X \not\in \mc{S}$.

One can see that given any
 $H \in [\kappa]^\kappa$,
 there are $X_1, X_2 \in [H]^\kappa$
 such that $X_1 \in \mc{S}$ and $X_2 \not\in \mc{S}$.
Hence, $\mc{S}$ is not Ramsey.
\end{proof}

On the other hand,
 we will show that if $\kappa$ is weakly compact,
 then every $\Sigma([\kappa]^2, [\kappa]^{<\kappa})$
 set is Ramsey.
We will use the following shrinking procedure,
 which we isolate here for clarity.

In the following setup,
 we do not actually need
 each $A \in [X]^n$ to have an associated $B_A$.
All we need is that for each
 $X' \in [X]^\kappa$,
 there is some $\alpha < \kappa$ such that
 $A := X' \cap \alpha$
 has an associated $B_A$.
However, we will not need this generality.

\begin{definition}
Let $X \in [\kappa]^\kappa$ and
 $\mc{Q}$ be a set of patterns.
Fix $n \in \omega$.
Suppose for each $A \in [X]^n$
 there is a set $B_A$ such that
 $(A,B_A) \in \mc{Q}$.
We say that $X$ is
 \defemph{fast} for
 $A \mapsto B_A$ iff
 for each $A \in [X]^n$,
 the only elements of $B_A \cap X$
 are $< \sup A$.
% if $a$ is the smallest element of $X$
% strictly greater than every element of $A$,
% then $a$ is strictly greater than every
% element of $B_A$.
\end{definition}

\begin{lemma}
\label{fastworks}
Let $X, \mc{Q}, n$ be as in the definition above,
 where each $A \in [X]^n$ has an associated $B_A$.
Suppose $X$ is fast for $A \mapsto B_A$.
Then every $X' \in [X]^\kappa$ matches
 some pattern in $\mc{Q}$.
\end{lemma}
\begin{proof}
Consider any $X' \in [X]^\kappa$.
Let $A \in [X']^n$ be the first $n$
 elements of $X'$.
Consider the set $B_A \cap X'$.
The only elements of $B_A \cap X$
 are $< \sup A$, so therefore
 the only elements of $B_A \cap X'$ are $< \sup A$.
On the other hand, the only elements of $X'$
 that are $< \sup A$ are the elements of
 $A$ themselves, and we have that
 $B_A \cap A = \emptyset$.
Thus, $B_A \cap X' = \emptyset$,
 which shows that $X'$ matches $(A,B_A)$.
\end{proof}

To produce an $X' \in [X]^\kappa$ that is fast for
 $A \mapsto B_A$, we shrink $X$ by subtracting the
 final parts of the $B_A$'s from $X$.

\begin{lemma}
\label{fastconstruction}
Let $X \in [\kappa]^\kappa$,
 $n \in \omega$,
 and $\mc{Q}$ be a set of
 $([\kappa]^n, [\kappa]^{<\kappa})$-patterns.
Assume that each $A \in [X]^n$
 has an associated $B_A$ such that
 $(A,B_A) \in \mc{Q}$.
Then there is some $X' \in [X]^\kappa$
 that is fast for
 $A \mapsto B_A$.
\end{lemma}
\begin{proof}
Fix a function $f : \kappa \to \kappa$ such that
 for each $\alpha$ and $A \in [\alpha]^n$,
 $\sup(B_A) < f(\alpha)$.
Thin down $X$ to produce an $X'$ that satisfies
 $f(A) < y$ for all $A \in [X']^n$ and $y \in X'$
 such that $A < y$.
This works.
% For simplicity,
%  we will consider the case that $n = 2$.
% To construct $X'$,
%  start with $X_0 := X$
%  and let $a_0 < a_1$ be its first two elements.
% Let $A = \{a_0, a_1\}$
%  and remove from $X_0$ every element of
%  $B_A$ strictly greater than $a_1$.
% Let $X_1 \subseteq X_0$ be the resulting set.
% Note that $a_0, a_1 \in X_1$,
%  and moreover these are its first two elements.
% Call the third element $a_2$.
% For each
%  $A = \{ a_i, a_2 \}$ for $i = 0,1$,
%  remove from $X_1$ every element of
%  $B_A$ that is strictly greater than $a_2$.
% Let $X_2 \subseteq X_1$ be the resulting set.
% Note that $a_0, a_1, a_2 \in X_2$.
% Continue this process transfinitely,
%  taking intersections at limit stages.
% Specifically, at a limit stage $\alpha < \kappa$,
%  define $X_\alpha := \bigcap_{\beta < \alpha}
%  X_\beta$ and let $a_\alpha$ be the $\alpha$-th
%  element of $X_\alpha$.
% The result is a sequence of sets
%  $X = X_0 \supseteq X_1 \supseteq X_2 \supseteq ...$
%  and $$X' := \bigcap_{\alpha < \kappa} X_\alpha =
%  \{ a_\alpha : \alpha < \kappa \}.$$
\end{proof}

Here is the promised result.

\begin{proposition}
\label{positiveweaklycompact}
Let $\kappa$ be a weakly compact cardinal.
Then every $\Sigma([\kappa]^2, [\kappa]^{<\kappa})$
 set is Ramsey.
\end{proposition}
\begin{proof}
Fix $\mc{S} \subseteq [\kappa]^\kappa$
 in $\Sigma([\kappa]^2, [\kappa]^{<\kappa})$.
Let $\mc{Q}$ be 
 a set of
 $([\kappa]^2, [\kappa]^{<\kappa})$-patterns
 which generate $\mc{S}$.
%That is,
% $$\mc{S} = \{ X \in [\kappa]^\kappa :
% X \mbox{ matches some } (A,B) \in \mc{Q} \}.$$
For each $A \in [\kappa]^2$,
 if there is some $B \in [\kappa]^{<\kappa}$
 such that $(A,B) \in \mc{Q}$,
 then let $B_A$ be some such $B$.
Let $c : [\kappa]^2 \to 2$ be the following coloring:
 $$c(A) :=
 \begin{cases}
 1 & \mbox{if } (A,B) \in \mc{Q}
 \mbox{ for some } B, \\
 0 & \mbox{otherwise.}
 \end{cases}$$
Since $\kappa$ is weakly compact,
 let $H \in [\kappa]^\kappa$
 be homogeneous for $c$.
That is, all pairs from $H$ are assigned
 the same color by $c$.
If $c``[H]^2 = \{0\}$,
 then no subset of $H$ can match any pattern
 from $\mc{Q}$, so we are done.

If $c``[H]^2 = \{1\}$,
 then each $A \in [H]^2$ has an associated $B_A$.
Apply Lemma~\ref{fastconstruction}
 to get a set $H' \in [H]^\kappa$
 that is fast for $A \mapsto B_A$.
By Lemma~\ref{fastworks},
 each $X \in [H']^\kappa$ matches a pattern in $\mc{Q}$.
\end{proof}

% In the proof above,
%  suppose that $\kappa$ is a measurable cardinal
%  with a normal ultrafilter $\mc{U}$.
% The set $H$ can be chosen to be inside $\mc{U}$.
% If each $A \in [H]^2$ has an associated $B_A$ such that
%  $(A,B_A) \in \mc{Q}$,
%  then by our comments following Lemma~\ref{fastconstruction},
%  we can construct a set $H' \subseteq H$ in $\mc{U}$
%  that is fast for $A \mapsto B_A$.
% Hence,
%  each $\Sigma([\kappa]^2,[\kappa]^{<\kappa})$
%  set is Ramsey as witnessed by a set in $\mc{U}$.

If $\kappa$ is a weakly compact cardinal,
 then we have in fact that for every $n \in \omega$,
 $\lambda < \kappa$,
 and $d : [\kappa]^n \to \lambda$,
 there is some $H \in [\kappa]^\kappa$ satisfying
 $|d``[H]^n| = 1$.
Thus, the argument from the proposition above
 yields the following.
It implies, in particular, that if $\kappa$
 is weakly compact, then every set in
 $\Sigma([\kappa]^n, [\kappa]^{<\kappa})$
 for $n \in \omega$ is Ramsey.
\begin{proposition}
\label{wcmanyanyonce}
Let $\kappa$ be weakly compact and let
 $1 \le \lambda < \kappa$.
Let $c : [\kappa]^\kappa \to (\lambda+1)$ be such that
 for each $\alpha < \lambda$,
 $c^{-1}(\alpha) \in \Sigma([\kappa]^n,[\kappa]^{<\kappa})$.
Then $c$ is Ramsey.
\end{proposition}
\begin{proof}
Note that we make no requirements on the complexity
 of $c^{-1}(\lambda)$.
For each $\alpha < \lambda$,
 let $\mc{Q}_\alpha$ be the set of
 $([\kappa]^n, [\kappa]^{<\kappa})$-patterns
 which generate $c^{-1}(\alpha)$.
For each $A \in [\kappa]^n$,
 if there is some $B \in [\kappa]^{<\kappa}$ such that
 $(A,B) \in \mc{Q}_\alpha$ for some $\alpha$,
 then let $B_A$ be some such $B$.
Note that if $(A,B_1) \in \mc{Q}_{\alpha_1}$ and
 $(A,B_2) \in \mc{Q}_{\alpha_2}$, then
 $\alpha_1 = \alpha_2$.
Let $d : [\kappa]^n \to (\lambda+1)$
 be the following coloring:
 $$d(A) := \begin{cases}
 \alpha & \mbox{if } (A,B) \in \mc{Q}_\alpha
 \mbox{ for some } B, \\
 \lambda & \mbox{otherwise.}
 \end{cases}$$
Since $\kappa$ is weakly compact,
 let $H \in [\kappa]^\kappa$ be such that
 $|d``[H]^n| = 1$.

If $d``[H]^n = \{\lambda\}$,
 then consider any $X \in [H]^\kappa$.
For each $A \in [X]^n$,
 there is no $B$ such that
 $(A,B) \in \mc{Q}_\alpha$ for some $\alpha < \lambda$.
Hence, $X$ is not in any
 $c^{-1}(\alpha)$ for $\alpha < \lambda$.
Thus, $X \in c^{-1}(\lambda)$.
This shows that $H$ is homogeneous for $c$.

The other case is that
 $d``[H]^n = \{ \alpha \}$ for some fixed $\alpha < \lambda$.
That is, for each $A \in [H]^n$,
 $(A,B_A) \in \mc{Q}_\alpha$.
Apply Lemma~\ref{fastconstruction}
 to get a set $H' \in [H]^\kappa$
 that is fast for $A \mapsto B_A$.
By Lemma~\ref{fastworks},
 each $X \in [H']^\kappa$
 matches a pattern in $\mc{Q}$.
%Apply the thinning procedure described in the proposition
% above to get a set $H' \in [H]^\kappa$ such that
% for each $A \in [H']^n$,
% if $a$ is the smallest element of $H'$
% strictly greater than the largest element of $A$,
% then $a$ is strictly greater than all elements of $B_A$.
%Now consider any set
% $X \in [H']^\kappa$.
%We will show that $c(X) = \alpha$,
% establishing that $H'$ is homogeneous for $c$.
%Let $A$ be the first $n$ elements of $X$.
%Let $a$ be the smallest element of $X$ above
% the largest element of $A$.
%Then $a$ is strictly greater than all elements of $B_A$.
%Hence, $X \cap B_A = \emptyset$,
% so $X$ matches $(A,B_A)$, and therefore
% $X \in c^{-1}(\alpha)$.
\end{proof}

\section{All $\Sigma([\kappa]^{<\omega}, [\kappa]^{<\kappa})$
 sets are Ramsey if $\kappa$ is Ramsey}
\label{section_ramsey}

The results in this section are analogous to those
 in the previous section, so we will only sketch the proofs.
Recall that $\kappa$ is a Ramsey cardinal iff given any
 $c : [\kappa]^{<\omega} \to 2$,
 there is some $H \in [\kappa]^\kappa$ such that for all
 $n \in \omega$,
 $|c``[H]^n| = 1$.
The following is analogous to
 Observation~\ref{fineenoughforpairs}:
\begin{observation}
\label{fineenoughfortuples}
For $A \in [\kappa]^n$,
 there is a $([\kappa]^n, [\kappa]^{<\kappa})$-pattern
 $(A,B)$ such that a set $X \in [\kappa]^\kappa$
 matches $(A,B)$ iff its first $n$ elements are
 the elements of $A$.
\end{observation}

We would like to say that if $\kappa$ is not a Ramsey
 cardinal, then there is some
 $\Delta([\kappa]^{<\omega}, [\kappa]^{<\kappa})$ set
 that is not Ramsey.
However, we know only the following assertion
 to be true:
\begin{proposition}
Let $\kappa$ be an infinite cardinal that is not Ramsey.
Then there are
 $\mc{S}_n \in \Delta([\kappa]^n, [\kappa]^{<\kappa})$
 for $n < \omega$ such that there is no
 $H \in [\kappa]^\kappa$ homogeneous for all
 $\mc{S}_n$.
\end{proposition}
\begin{proof}
Let $c : [\kappa]^{<\omega} \to 2$ witness that $\kappa$
 is not Ramsey.
Using the observation above, for each $n \in \omega$,
 define $\mc{S}_n$ so that given any $X \in [\kappa]^\kappa$,
 $X \in \mc{S}$ iff the first $n$ elements of $X$
 are colored $1$ by $c$.
If $H \in [\kappa]^\kappa$ is a set which is homogeneous
 for each $\mc{S}_n$, then
 $|c``[H]^n| = 1$ for each $n$, which is a contradiction.
\end{proof}

The following is a straightforward modification
 of Proposition~\ref{positiveweaklycompact}:

\begin{proposition}
\label{ramseysaregood}
Let $\kappa$ be a Ramsey cardinal.
Then every
 $\Sigma([\kappa]^{<\omega}, [\kappa]^{<\kappa})$ set
 is Ramsey.
\end{proposition}
\begin{proof}
Fix $\mc{S} \subseteq [\kappa]^\kappa$
 in $\Sigma([\kappa]^{<\omega}, [\kappa]^{<\kappa})$.
Let $\mc{Q}$ be the set of patterns which
 generate $\mc{S}$.
For each $A \in [\kappa]^{<\omega}$, if there is some
 $B \in [\kappa]^{<\kappa}$ such that
 $(A,B) \in \mc{Q}$, then let $B_A$ be some such $B$.
For each $n \in \omega$,
 let $c_n : [\kappa]^n \to 2$ be the following
 coloring:
 $$ c_n(A) := \begin{cases}
 1 & \mbox{if } (A,B) \in \mc{Q}
 \mbox{ for some } B, \\
 0 & \mbox{otherwise.}
 \end{cases}$$
Since $\kappa$ is a Ramsey cardinal,
 let $H \in [\kappa]^\kappa$
 simultaneously homogenize each $c_n$.

There are two cases.
The first case is that for all $n \in \omega$,
 $c_n``[H]^n = \{0\}$.
When this happens,
 no $X \in [H]^\kappa$
 can match any pattern
 $(A,B) \in \mc{Q}$,
 so $H$ is homogeneous for $\mc{S}$.

The other case is that there is some fixed
 $n \in \omega$ such that
 $c_n``[H]^n = \{1\}$.
Each $A \in [H]^n$ has an associated
 $B_A$.
Apply Lemma~\ref{fastconstruction}
 to get a set $H' \in [H]^\kappa$
 that is fast for $A \mapsto B_A$.
By Lemma~\ref{fastworks},
 each $X \in [H']$ matches a pattern in $\mc{Q}$.
%When this happens, 
% apply the thinning procedure described in
% Proposition~\ref{positiveweaklycompact}
% to get a set $H' \in [H]^\kappa$ such that
% for each $A \in [H']^n$,
% if $a$ is the smallest element of $H'$
% strictly greater than all elements of $A$,
% then $a$ is strictly greater than all elements of
% $B_A$.
%Consider any $X \in [H']^\kappa$.
%Let $A \in [X]^n$ be the first $n$ elements
% of $X$.
%We have that $(A,B_A) \in \mc{Q}$.
%Just like in Proposition~\ref{positiveweaklycompact},
% we have that $B_A \cap X = \emptyset$.
%Thus, $X$ matches $(A,B_A)$,
% so $X \in \mc{S}$.
\end{proof}

If $\kappa$ is a Ramsey cardinal,
 then for any cardinal $\lambda < \kappa$,
 for any coloring
 $d : [\kappa]^{<\omega} \to \lambda$,
 there is set $H \in [\kappa]^\kappa$ such that
 for all $n < \omega$,
 $|d ``[H]^n| = 1$.
This gives us the following:
\begin{proposition}
\label{ramseymulti}
Let $\kappa$ be Ramsey
 and let $1 \le \lambda < \kappa$.
Let $c : [\kappa]^\kappa \to (\lambda + 1)$
 be such that for each $\alpha < \lambda$,
 $c^{-1}(\alpha) \in
 \Sigma([\kappa]^{<\omega}, [\kappa]^{<\kappa})$.
Then $c$ is Ramsey.
\end{proposition}
\begin{proof}
The proof is analogous to
 Proposition~\ref{wcmanyanyonce}.
For each $\alpha < \lambda$,
 let $\mc{Q}_\alpha$ be the set of patterns
 which generate $c^{-1}(\alpha)$.
We let $d : [\kappa]^{<\omega} \to (\lambda + 1)$
 be such that $d(A) := \alpha$ if
 $(A,B) \in \mc{Q}_\alpha$ for some $B$,
 and $d(A) := \lambda$ otherwise.
Note that $d$ is well-defined.
Since $\kappa$ is Ramsey,
 let $H \in [\kappa]^\kappa$ be such that
 $|d``[H]^n| = 1$ for all $n \in \omega$.

There are two cases.
The first case is that
 $d``[H]^n = \{ \lambda \}$ for all $n$.
In this case, it can be argued that
 each $X \in [H]^\kappa$ is in
 $d^{-1}(\lambda)$.
The other case is that
 $d``[H]^n = \{ \alpha \}$ for some fixed
 $n < \omega$ and $\alpha < \lambda$.
In this case, $H$ can be shrunk as before to
 produce $H' \in [H]^\kappa$
 with the property that each
 $X \in [H']^\kappa$ is in $c^{-1}(\alpha)$.
\end{proof}

\section{All $\Sigma([\kappa]^{<\omega},
 \p(\kappa) - \mc{U})$
 sets are Ramsey if $\mc{U}$ is a 
 $\kappa$-complete ultrafilter}
\label{section_measurable}

So far, we have said little about
 patterns $(A,B)$ where $|B| = \kappa$.
In this section, we will show that when
 $\kappa$ is a measurable cardinal
 and when we fix a
 $\kappa$-complete ultrafilter on $\kappa$,
 sets $B$ not in the ultrafilter are small enough
 to be used in patterns $(A,B)$
 that will still generate Ramsey sets.
%What makes these sets $B$
% special is that they can be used
% to generalize Lemma~\ref{fastconstruction}.
Recall that an ultrafilter $\mc{U}$
 is $\kappa$-\textit{complete} iff it is closed under
 intersections of size $< \kappa$.
An ultrafilter on $\kappa$
 is \textit{normal} iff it is $\kappa$-complete
 and moreover is closed under
 \textit{diagonal} intersections.

\begin{theorem}
\label{measurable_theorem}
Let $\kappa$ be a measurable cardinal
 and let $\mc{U}$ be a normal ultrafilter
 on $\kappa$.
Then every $\Sigma([\kappa]^{<\omega},
 \p(\kappa) - \mc{U})$ set is Ramsey,
 as witnessed by a set $H \in \mc{U}$.
\end{theorem}
\begin{proof}
Fix $\mc{S}$ in $\Sigma([\kappa]^{<\omega},
 \p(\kappa) - \mc{U})$,
 and let $\mc{Q}$ be the set of
 $([\kappa]^{<\omega},
 \p(\kappa) - \mc{U})$-patterns
 which generates it.
For each $A \in [\kappa]^{<\omega}$,
 let $C_A \in \p(\kappa) - \mc{U}$
 be some set $B$ such that
 $(A,B) \in \mc{Q}$ if such a $B$ exists,
 and let $C_A = \emptyset$ otherwise.

For each $\alpha < \kappa$,
 let $Y_\alpha = \bigcap \{
 \kappa - C_A : \max{A} = \alpha \} \in \mc{U}$.
Let $Y$ be the diagonal intersection
 of these $Y_\alpha$'s:
 $Y = \{ \beta :
 \beta \in \bigcap_{\alpha < \beta} Y_\alpha \}$,
 which is in $\mc{U}$ because $\mc{U}$ is normal.
Suppose temporarily that
 $A \in [Y]^{<\omega}$, $y \in Y$, and $A < y$.
Let $\alpha = \max A$, so $\alpha < y$.
Since $y \in Y$, by definition we have
 $y \in Y_\alpha$.
This implies that $y \in \kappa - C_A$.
Hence, $y \not\in C_A$.

Now let $c : [Y]^{<\omega} \to 2$ be the coloring
 given by $c(A) = 1$ if $(A,C_A) \in \mc{Q}$,
 and $c(A) = 0$ otherwise.
Since $Y \in \mc{U}$ and $\mc{U}$ is $\kappa$-complete,
 there is some $H \in [Y]^\kappa$ in $\mc{U}$
 that is homogeneous for $c$.
If $c``[H]^n = \{0\}$ for all $n$,
 then no $X \in [H]^\kappa$ matches a pattern in
 $\mc{Q}$, and we are done.
If $c``[H]^n = \{1\}$ for some fixed $n$,
 then consider any $X \in [H]^\kappa$.
Let $A$ be the first $n$ elements of $X$.
By what we said above, any element of $Y$
 greater than $\max A$ is not in $C_A$.
Hence, every element of $X$ greater than $\max A$
 is not in $C_A$.
This shows that $X \cap C_A = \emptyset$.
Thus, $X$ matches the pattern $(A,C_A) \in \mc{Q}$.
\end{proof}

If $\mc{U}$ is not a normal ultrafilter
 in the above theorem
 but only a $\kappa$-complete ultrafilter,
 then we have the weaker conclusion
 that $H \in [\kappa]^\kappa$.
This can be proved by modifying
 Lemma~\ref{fastconstruction}.

\section{Not all
 $\Delta([\kappa]^{\omega}, [\kappa]^{<\kappa}$)
 sets are Ramsey if $\kappa > \omega$}

It is well known that assuming the Axiom of Choice,
 not every subset of $[\omega]^\omega$ is Ramsey.
Since $[\omega]^\omega = \Delta(
 [\omega]^{\omega}, [\omega]^{\le \omega})$,
 we have that not every
 $\Delta(
 [\omega]^{\omega}, [\omega]^{\le \omega})$
 set is Ramsey.
In this section,
 we will show that the argument
 for $[\omega]^\omega$
 shows that when $\kappa > \omega$,
 not every
 $\Sigma([\kappa]^\omega, [\kappa]^{<\kappa})$
 set is Ramsey.
%Afterwards, it should be clear to the reader
% how to prove there is a $\Delta(
% [\omega]^{\omega}, [\omega]^{\le \omega})$
% set that is not Ramsey.

\begin{observation}
Let $\kappa > \omega$
 be a cardinal.
For $A \in [\kappa]^\omega$,
 there is a
 $([\kappa]^\omega, [\kappa]^{<\kappa})$-pattern
 $(A,B)$ such that a set $X \in [\kappa]^\kappa$
 matches $(A,B)$ iff the first $\omega$ elements
 of $X$ are
 the elements of $A$.
\end{observation}

Given sets $A, B \in [\kappa]^\kappa$,
 recall that $A \Delta B$ is the set
 $(A - B) \cup (B - A)$.
This next proposition uses the Axiom of Choice.

\begin{proposition}
\label{acmakesitfail}
Let $\kappa > \omega$ be a cardinal.
There is a $\Delta([\kappa]^{\omega}, [\kappa]^{<\kappa})$
 set that is not Ramsey.
\end{proposition}
\begin{proof}
Given a set $X \in [\kappa]^\kappa$,
 let $X'$ be the set of the first $\omega$ elements of $X$.
Given $X_1, X_2 \in [\kappa]^\kappa$,
 we write $X_1 \equiv X_2$ iff
 1) $\sup X_1' = \sup X_2'$ and
 2) $|X_1' \Delta X_2'| < \omega$.
Using the Axiom of Choice,
 we may pick a representative from each
 $\equiv$-equivalence class.
Let $\mc{S} \subseteq [\kappa]^\kappa$
 be defined such that for each
 $X \in [\kappa]^\kappa$,
 $X \in \mc{S}$ iff $|X' \Delta Y'|$ is even,
 where $Y$ is the representative from
 $X$'s $\equiv$-equivalence class.
Now, given any $X_1 \in [\kappa]^\kappa$,
 there is some $X_2 \in [X_1]^\kappa$ such that
 $X_1 \in \mc{S}$ iff $X_2 \not\in \mc{S}$:
 to produce such an $X_2$, simply remove the first
 element from $X_1$.
\end{proof}

\section{Constructible Patterns}
\label{section_constructible}

We mentioned that,
 assuming the Axiom of Choice,
 there is a subset
 of $[\omega]^\omega$
 that is not Ramsey.
However,
 if $\mc{S} \subseteq [\omega]^\omega$
 is in $L(\mbb{R})$
 and we assume there are large cardinals
 in the universe,
 then $\mc{S}$ is Ramsey
 \cite{kan}.
With the same large cardinal assumptions,
 Martin showed \cite{kan}
 that every $\mc{S} \subseteq [\omega_1]^{\omega_1}$
 in $L(\mbb{R})$ is Ramsey
 from the point of view of $L(\mathbb{R})$.
In this section,
 we show results of a similar flavor:
 if the set of patterns $\mc{Q}$ used to generate a set
 is not too complicated, then the set $\mc{S}$
 generated in the full universe
 must be Ramsey.

Recall that if $0^\#$ exists,
 then there is a proper class of indiscernibles
 $\mc{I} \subseteq \mbox{Ord}$,
 called \textit{Silver indiscernibles}, such that $L$
 is the Skolem hull of $\mc{I}$.
Given a cardinal $\kappa$,
 let $\mc{I}_\kappa$ refer to $\kappa \cap \mc{I}$.
%The idea of the argumet to follow is that
% a pattern must be simple in order for
% exactly half of the indiscernibles less
% than a cardinal to match it.

\begin{lemma}
\label{finite_lemma}
Let $A \subseteq \mc{I}$ be in $L$.
Then $A$ is finite.
\end{lemma}
\begin{proof}
Given any countably infinite subset $C$ of $\mc{I}$
 and $\alpha \in \mc{I}$ satisfying $\sup C \le \alpha$,
 $0^\#$ is the theory of $L_\alpha$
 with constant symbols
 for the elements of $C$.
If $A$ is infinite, then within $L$ we can define
 $0^\#$, which is impossible.
\end{proof}

We must now deal with the $\mc{B}$ components
 of our patterns.
\begin{definition}
Assume $0^\#$ exists.
Let $\kappa > \omega$ be a cardinal.
Let $B \subseteq \kappa$ be in $L$.
We call $B$ \defemph{bad} iff
 $\mc{I}_\kappa - B$ has size $< \kappa$.
We call $B$ \defemph{good} iff
 $\mc{I}_\kappa \cap B$ has size $< \kappa$.
\end{definition}

If $B$ is bad, then no
 $X \in [\mc{I}_\kappa]^\kappa$ can match $(A,B)$
 for any $A$.
\begin{lemma}
Assume $0^\#$ exists.
Let $\kappa > \omega$ be a cardinal.
Let $B \subseteq \kappa$ in $L$ be not bad.
Then $B$ is good.
\end{lemma}
\begin{proof}
Since $0^\#$ exists, let $\alpha_0 < ... < \alpha_l < \kappa$
 be indiscernibles such that whenever $\beta_1$ and $\beta_2$
 are between two consecutive elements of
 $$0, \alpha_0, ..., \alpha_l, \kappa,$$
 then $\beta_1 \in B \cap \mc{I}$
 iff $\beta_2 \in B \cap \mc{I}$.
The set $(\alpha_l, \kappa) \cap \mc{I}_\kappa$
 is either a subset of $B$ or
 disjoint from $B$.
It cannot be a subset of $B$ because then we would have
 that $\mc{I}_\kappa - B$ has size $< \kappa$,
 meaning $B$ is bad.
So it must be disjoint from $B$,
 and therefore $B$ is good.
\end{proof}

We now have that if $\mc{Q} \subseteq L$ is a set
 of patterns and $X \in [\mc{I}_\kappa]^\kappa$
 matches some $(A,B) \in \mc{Q}$, then $A$ is finite
 and $B$ is good.
Hence, the $(A,B)$ that we must consider are essentially
 $([\kappa]^{<\omega}, [\kappa]^{<\kappa})$-patterns:

However, this does not imply that the set
 $\mc{S}$ generated by $\mc{Q}$
 is Ramsey.
The problem is Observation~\ref{fineenoughforpairs},
 which in a more precise form gives us that
 for each $A \in [\kappa]^2$,
 there is some $B \in [\kappa]^{<\kappa}$ such that
 $(A,B) \in L$ and for any
 $X \in [\kappa]^\kappa$,
 $X$ matches $(A,B)$ iff its
 first two elements are the elements of $A$.
This gives us the following:
\begin{observation}
Let $\kappa$ be an infinite
 cardinal that is not weakly compact.
Then there is a set $\mc{Q} \subseteq L$ of
 $([\kappa]^2, [\kappa]^{<\kappa})$-patterns such that
 the set $\mc{S} \subseteq [\kappa]^\kappa$
 generated by $\mc{Q}$
 is not Ramsey.
\end{observation}

A similar situation occurs when,
 more generally,
 $\kappa$ is not a Ramsey cardinal.
On the other hand, we have the following:
\begin{proposition}
Let $\kappa > \omega$ be a Ramsey cardinal.
Let $\mc{Q} \subseteq L$ be a set of patterns.
Then the set
 $\mc{S} \subseteq [\kappa]^\kappa$
 generated by $\mc{Q}$
 is Ramsey.
\end{proposition}
\begin{proof}
Since $\kappa$ is a Ramsey cardinal,
 $0^\#$ exists.
Consider $\mc{I}_\kappa$.
Let $\mc{Q}' \subseteq \mc{Q}$
 be the set of $(A,B) \in \mc{Q}$ such that
 $A$ is finite and $B$ is good.
By the previous lemmas,
 for each $X \in [\mc{I}_\kappa]^\kappa$,
 we have $X \in \mc{S}$ iff $X$ is in the
 set generated by $\mc{Q}'$.
Thus, it suffices to find a set $H \in
 [\mc{I}_\kappa]^\kappa$ that is homogeneous
 for the set generated by $\mc{Q}'$.
For each $n \in \omega$,
 let $c_n : [\kappa]^n \to 2$
 be the coloring defined by
 $c_n(A) := 1$ if $(A,B) \in \mc{Q}'$ for some $B$,
 and $c_n(A) := 0$ otherwise.
Since $\kappa$ is Ramsey,
 let $H \in [\mc{I}_\kappa]^\kappa$
 homogenize each $c_n$.
If $c_n``[H]^n = \{0\}$ for each $n$,
 then no $X \in [H]^\kappa$ matches a pattern
 in $\mc{Q}'$.
On the other hand, suppose $c_n``[H]^n = \{1\}$ for
 some fixed $n$.
Then we may apply the usual shrinking procedure,
 since each $B$ under consideration is good,
 to produce $H' \in [H]^\kappa$ such that
 every $X \in [H']^\kappa$ matches a pattern in
 $\mc{Q}'$.
\end{proof}

Here is another way to ensure that the set
 generated by $\mc{Q} \subseteq L$ is Ramsey:

\begin{proposition}
\label{zero_sharp_prop}
Assume $0^\#$ exists.
Let $\kappa > \omega$ be a cardinal.
Let $\mc{Q} \in L$ be a set of patterns.
Then the set $\mc{S} \subseteq [\kappa]^\kappa$
 generated by $\mc{Q}$ is Ramsey.
\end{proposition}
\begin{proof}
Suppose $\mc{Q} = \rho(\vec{\alpha}_0, \vec{\alpha}_1)$,
 where $\rho$ is a Skolem term and
 $\vec{\alpha}_0, \vec{\alpha}_1$ are finite increasing
 sequences of elements of $\mc{I}$ such that
 $\mbox{max}(\vec{\alpha}_0) < \kappa \le
 \mbox{min}(\vec{\alpha}_1)$.
Let $I = \mc{I}_\kappa \cap
 (\mbox{max}(\vec{\alpha}_0), \kappa)$.
Let $J \in [I]^\kappa$ be such that
 between any two elements of $J$ there
 are infinitely many elements of $I$,
 and there are infinitely many elements of $I$
 before the first element of $J$.
We will show that either
 $[I]^\kappa \cap \mc{S} = \emptyset$
 or $[J]^\kappa \subseteq \mc{S}$.

Suppose there is some fixed
 $X \in [I]^\kappa \cap \mc{S}$.
Let $(A,B) \in \mc{Q}$ be such that
 $X \in [A;B]$.
Because $A \subseteq X \subseteq I$,
 by Lemma~\ref{finite_lemma}
 $A$ is finite.
Since $B \in L$,
 let $B = \tau(\vec{\beta}_0, \vec{\beta}_1, \vec{\beta}_2)$
 where $\tau$ is a Skolem term and
 $\vec{\beta}_0, \vec{\beta}_1, \vec{\beta}_2$
 are finite increasing sequences of elements of
 $\mc{I}$ such that
 $$\mbox{max}(\vec{\beta}_0) \le
 \mbox{max}(\vec{\alpha}_0) <
 \mbox{min}(\vec{\beta}_1) \le
 \mbox{max}(\vec{\beta}_1) <
 \kappa \le
 \mbox{min}(\vec{\beta}_2).$$
Assume that all elements of $A$ occur in $\vec{\beta}_1$.
%Now some elements of $\vec{\beta}_1$ might be in $A$,
% which causes a slight complication,
% and this is the need for $J$ to be sparse.
Enumerate $\vec{\beta}_1$ in increasing order as
 $\vec{\beta}_1 = \langle \beta_1^i : i < n \rangle$.
Let $F \subseteq n$ be such that
 $A = \{ \beta_1^i : i \in F \}$.

Now fix $Y \in [J]^\kappa$.
We must show that $Y \in \mc{S}$.
That is, we must find $(A',B') \in \mc{Q}$
 such that $Y \in [A';B']$.
Let $A'$ be the first $|F|$ elements of $Y$.
Enumerate $A'$ as $A' = \{ \gamma^i \in J : i \in F \}$.
We now must enlarge $A'$ to get a set of size $n$.
Let $\gamma^i \in I$ for $i \in n - F$ be such that
 the sequence $\vec{\gamma} =
 \langle \gamma^i \in I : i < n \rangle$
 is strictly increasing and
 $\gamma^{n-1} < \mbox{min}(Y - A')$.
This is possible because $J$ is sparse enough.
Now let $B' = \tau(\vec{\beta}_0, \vec{\gamma}, \vec{\beta_2})$.
It remains to show that $(A',B') \in \mc{Q}$
 and $Y \in [A',B']$.

Since $(A,B) \in \mc{Q}$, we have
 $$( \{ \beta_1^i : i \in F \},
 \tau(\vec{\beta}_0, \vec{\beta}_1, \vec{\beta}_2 )) \in
 \rho( \vec{\alpha}_0, \vec{\alpha}_1 ).$$
By indiscernibility, we have
 $$( \{ \gamma^i : i \in F \},
 \tau(\vec{\beta}_0, \vec{\gamma}, \vec{\beta}_2 )) \in
 \rho( \vec{\alpha}_0, \vec{\alpha}_1 ).$$
That is, $(A',B') \in \mc{Q}$.

%Finally, let $f = \mbox{max}(F)$.
Because $X \subseteq I$,
 there is some element of
 $I \cap (\beta_1^{n-1}, \kappa)$ not in $B$.
So by indiscernibility,
 no element of $I \cap (\beta_1^{n-1}, \kappa)$
 is in $B$.
Again by indiscernibility,
 no element of $I \cap (\gamma^{n-1}, \kappa)$
 is in $B'$.
However,
 $Y - A' \subseteq I \cap (\gamma^{n-1}, \kappa)$,
 because $\gamma^{n-1}$ is
 $< \mbox{min}(Y - A')$.
Because also $A' \cap B' = \emptyset$,
 we have that $Y \cap B' = \emptyset$.
%Also, $A' \subseteq Y$.
This establishes that
 $Y \in [A';B']$.
%For every $\delta \in B$, we have
% $\delta \not\in X$.
%In particular, we have
% $$\delta \not\in \tau(\vec{\beta}_0,
% \vec{\beta}_1, \vec{\beta}_2)
% \mbox{ for all }
%5 \delta \in I \cap (\mbox{max}(\beta_1^f), \kappa).$$
%By indiscernibility, we have
% $$\delta \not\in \tau(\vec{\beta}_0,
% \vec{\gamma}, \vec{\beta}_2)
% \mbox{ for all }
% \delta \in (\mbox{max}(\gamma_1^f), \kappa)$$
%That is, no element of
\end{proof}

This next question is natural along our
 line of inquiry:
\begin{question}
Does it follow from large cardinals,
 or is it even consistent with the Axiom of Choice,
 that for every set $\mc{Q} \in L(\mbb{R})$
 of $([\omega_1]^{<\omega_1}, [\omega_1]^{<\omega_1})$-patterns,
 the set generated by $\mc{Q}$ is Ramsey?
\end{question}

% Finally, let us address sets $\mc{S} \subseteq
%  [\omega_1]^{\omega_1}$ generated by patterns in $L(\mbb{R})$.
% Given $\mc{F} \subseteq [\omega_1]^{<\omega_1}$,
%  say that $X \in [\omega_1]^{\omega_1}$ is \textit{caught}
%  by $\mc{F}$ iff $(\exists \alpha < \omega_1)\,
%  X \cap \alpha \in \mc{F}$.
% Given a set $\mc{Q} \in L(\mbb{R})$ of
%  $([\omega_1]^{<\omega_1}, [\omega_1]^{<\omega_1})$-patterns,
%  there are sets $\mc{F}^+, \mc{F}^- \subseteq
%  [\omega_1]^{<\omega_1}$ in $L(\mbb{R})$ such that
%  each set $X \in [\omega_1]^{\omega_1}$ is caught by either
%  $\mc{F}^+$ or $\mc{F}^-$ but not both, and
%  the set generated by $\mc{Q}$ is the collection of
%  sets caught by $\mc{F}^+$.

% \begin{lemma}
% Let $\mc{F} \subseteq [\omega_1]^{<\omega_1}$
%  and $H \in [\omega_1]^{\omega_1}$ both be in $L(\mbb{R})$.
% Assume that $L(\mbb{R})$ satisfies
%  $(\forall X \in [H]^{\omega_1})\,X$
%  is not caught by $\mc{F}$.
% Then the full universe satisfies this as well.
% \end{lemma}
% \begin{proof}
% Assume towards a contradiction that there is some
%  $X \in [H]^{\omega_1}$ in the full universe
%  that is caught by $\mc{F}$.
% Let $\alpha < \omega_1$ be such that
%  $X \cap \alpha \in \mc{F}$.
% Let $Y = (X \cap \alpha) \cup (H - \alpha)$.
% This set is also caught by $\mc{F}$,
%  it is in $L(\mbb{R})$,
%  and it is a size $\omega_1$ subset of $H$.
% \end{proof}

\section{Acknowledgements}

I would like the thank Andreas Blass
 for discussions on this project.
The referee also simplified and improved
 several theorems,
 in particular Theorem~\ref{theorem2},
 Theorem~\ref{measurable_theorem},
 and Proposition~\ref{zero_sharp_prop}.

\end{document}